\pdfoutput=1
\documentclass[11pt,a4paper]{amsart}
\usepackage[utf8]{inputenc}

\calclayout

\usepackage{mathtools}
\mathtoolsset{showonlyrefs}

\usepackage{color}
\usepackage[colorlinks=true,urlcolor=blue,anchorcolor=blue,citecolor=blue,filecolor=blue,linkcolor=blue,menucolor=blue,linktocpage=true,pdfa=true,unicode=true,bookmarksopen=true]{hyperref}

\usepackage{mlmodern}
\usepackage[T1]{fontenc}
\usepackage{bbm}
\DeclareMathAlphabet{\mathbfi}{OML}{cmm}{b}{it}

\newtheorem{theorem}{Theorem}
\newtheorem{lemma}[theorem]{Lemma}
\newtheorem{proposition}[theorem]{Proposition}

\let\originalleft\left
\let\originalright\right
\renewcommand{\left}{\mathopen{}\mathclose\bgroup\originalleft}
\renewcommand{\right}{\aftergroup\egroup\originalright}

\newcommand{\total}{\mathop{}\!\mathrm{d}}
\newcommand{\1}{\mathbbm{1}}
\newcommand{\Tr}{\operatorname{Tr}}
\newcommand{\eqend}[1]{\,#1}
\newcommand{\M}{\mathcal{M}}

\newcommand{\stp}{_{*,+}}

\makeatletter
\def\abs{\@ifnextchar[\abs@size\abs@nosize}
\def\abs@size[#1]#2{\mathopen{#1\lvert}{#2}\mathclose{#1\rvert}}
\def\abs@nosize#1{\left\lvert{#1}\right\rvert}
\def\norm{\@ifnextchar[\norm@size\norm@nosize}
\def\norm@size[#1]#2{\mathopen{#1\lVert}{#2}\mathclose{#1\rVert}}
\def\norm@nosize#1{\left\lVert{#1}\right\rVert}
\makeatother

\allowdisplaybreaks

\begin{document}

\title{The Powers--St{\o}rmer--Ogata inequality for operator monotone functions}

\author{Markus B. Fr{\"o}b}
\address{Department Mathematik, Friedrich-Alexander-Universit{\"a}t Erlangen-N{\"u}rnberg, \hfill\linebreak Cauerstra{\ss}e 11, 91058 Erlangen, Germany}
\email{markus.froeb@fau.de}

\thanks{This work has been funded by the Deutsche Forschungsgemeinschaft (DFG, German Research Foundation) --- project no. 574438490 (``Modular Hamiltonian and relative entropy in interacting quantum field theories'').}

\subjclass[2020]{Primary 46L10; Secondary 47A63, 26A48}

\date{15. November 2026}

\begin{abstract}
We show that for any positive operator monotone function $f$ the inequality
\begin{equation}
\int_{[0,\infty)} f(t) \total \norm{ E^{\Delta_{\varphi,\psi}}(t) \xi_\psi }^2 + f'_\infty \varphi\left( \1 - s(\psi) \right) \geq \frac{f(1)}{2} \Bigl( \varphi(\1) + \psi(\1) - \norm{ \varphi - \psi } \Bigr)
\end{equation}
holds, where $\varphi, \psi \in \M\stp$ are two normal positive linear functionals on a von Neumann algebra $\M$, and $\Delta_{\varphi,\psi}$ is the associated relative modular operator. Choosing $f(t) = t^s$ with $s \in [0,1]$, the Powers--St{\o}rmer--Ogata inequality is recovered.
\end{abstract}

\maketitle

Given a von Neumann algebra $\M$ and two normal positive linear functionals $\varphi, \psi \in \M\stp$, the inequality
\begin{equation}
\label{eq:vn-s}
2 \norm[\Big]{ \Delta_{\varphi,\psi}^\frac{s}{2} \xi_\psi }^2 \geq \varphi(\1) + \psi(\1) - \norm{ \varphi - \psi }
\end{equation}
has been shown by Powers and St{\o}rmer~\cite{powersstormer1970} for $s = 1/2$, and extended to $s \in [0,1]$ by Ogata~\cite{ogata2011}. Here, $\Delta_{\varphi,\psi}$ is the relative modular operator associated with $\varphi$ and $\psi$, and $\xi_\psi \in \mathcal{P}$ is the unique representative vector of $\psi$ in the natural positive cone $\mathcal{P}$. The analogous inequality for matrices
\begin{equation}
\label{eq:matrix-s}
2 \Tr\left( A^s B^{1-s} \right) \geq \Tr\left( A + B - \abs{A-B} \right)
\end{equation}
with $A, B \in \mathbb{C}^{n \times n}$, $A, B \geq 0$ was first shown by Audenaert et al.~\cite{audenaertetal2007}, and a simplified proof was given by Ozawa~\cite[Prop.~1.1]{jaksicetal2012}. The matrix inequality~\eqref{eq:matrix-s} has been generalized to
\begin{equation}
\label{eq:matrix-f}
2 \Tr\left( f(A) \frac{B}{f(B)} \right) \geq \Tr\left( A + B - \abs{A-B} \right)
\end{equation}
by Dinh, Osaka and Toan~\cite[Thm.~2.1]{dinhosakatoan2013}, where $f \colon (0,\infty) \to (0,\infty)$ is any $2n$-matrix monotone function and $B/f(B)$ is defined to be zero on the kernel of $B$. Choosing $f(t) = t^s$, which is well-known to be operator monotone for $s \in [0,1]$, one recovers the inequality~\eqref{eq:matrix-s}. The class of functions for which~\eqref{eq:matrix-f} is valid can even be extended further~\cite{dinhosaka2015,kianetal2023}.

An important application of the Powers--St{\o}rmer--Ogata inequality arises in quantum hypothesis testing, where one seeks to distinguish two quantum states by measurements. In this context, the matrix inequality~\eqref{eq:matrix-s} provides an upper bound on the optimal error probability for discriminating two states, and plays a central role in the derivation of the quantum Chernoff bound~\cite{audenaertetal2007}. Its generalization to arbitrary von Neumann algebras allowed such estimates to be applied to infinite-dimensional quantum systems, including non-equilibrium quantum statistical mechanics~\cite{jaksicetal2012}. Improving the inequality then yields sharper bounds on the optimal error probability for state discrimination: minimizing the left-hand side of~\eqref{eq:matrix-f} over all positive $2n$-matrix monotone functions $f$ yields a smaller value than minimizing the left-hand side of~\eqref{eq:matrix-s} over $s \in [0,1]$. It is thus desirable to obtain a similar generalization for general von Neumann algebras.

In this short note, we will show
\begin{theorem}
\label{thm:ineq}
Let $(\M, \mathcal{H}, J, \mathcal{P})$ be a von Neumann algebra in standard form~\cite{haagerup1975}, and let $f \colon [0,\infty) \to [0,\infty)$ be an operator monotone function with $f'_\infty \coloneq \lim_{t \to \infty} f(t)/t$. For any two normal positive linear functionals $\varphi, \psi \in \M\stp$, the inequality
\begin{equation}
\int_0^\infty f(t) \total \norm{ E^{\Delta_{\varphi,\psi}}(t) \xi_\psi }^2 + f'_\infty \varphi\left( \1 - s(\psi) \right) \geq \frac{f(1)}{2} \Bigl( \varphi(\1) + \psi(\1) - \norm{ \varphi - \psi } \Bigr)
\end{equation}
holds, where $s(\psi) \in \M$ is the support projection of $\psi$, $\xi_\psi \in \mathcal{P}$ is the unique representative vector of $\psi$ in the natural positive cone, $\Delta_{\varphi,\psi}$ is the relative modular operator associated with $\varphi$ and $\psi$, and $E^{\Delta_{\varphi,\psi}}$ its spectral measure.
\end{theorem}
The Powers--St{\o}rmer--Ogata inequality~\eqref{eq:vn-s} follows as in the matrix case by taking $f(t) = t^s$. The proof of the theorem involves a lemma of Kosaki:
\begin{lemma}[{\cite[Lemma~1.1]{kosaki1986}}]
\label{lemma:inf}
In the setting of Thm.~\ref{thm:ineq}, for each $t > 0$ it holds that
\begin{equation}
\left( \xi_\psi, \frac{\Delta_{\varphi,\psi}}{t + \Delta_{\varphi,\psi}} \xi_\psi \right) = \inf_{x \in \M} \left[ \psi\left( (\1-x^*) (\1-x) \right) + \frac{1}{t} \varphi\left( x x^* \right) \right] \eqend{.}
\end{equation}
\end{lemma}
In addition, we employ one of the integral representations of operator monotone functions due to Löwner:
\begin{theorem}[{\cite[Thm.~4.9]{hansen2013}}]
\label{thm:mon}
For any operator monotone function $f \colon [0,\infty) \to [0,\infty)$, there exists a bounded positive measure $\mu \colon [0,1] \to [0,\infty)$ such that
\begin{equation}
f(t) = \int_{(0,1)} \frac{t}{\lambda + (1-\lambda) t} \total \mu(\lambda) + \mu(\{0\}) + t \mu(\{1\}) \eqend{.}
\end{equation}
\end{theorem}
Finally, we need the following bound:
\begin{proposition}
\label{prop:ineq}
In the setting of Thm.~\ref{thm:ineq}, for any $x \in \M$ and any $\lambda \in (0,1)$ it holds that
\begin{equation}
\frac{\varphi\left( x x^* \right)}{\lambda} + \frac{\psi\left( (\1-x^*) (\1-x) \right)}{1-\lambda} \geq \frac{1}{2} \Bigl( \varphi(\1) + \psi(\1) - \norm{ \varphi - \psi } \Bigr) \eqend{.}
\end{equation}
\end{proposition}
Thm.~\ref{thm:ineq} then follows easily:
\begin{proof}[Proof of Thm.~\ref{thm:ineq}]
Since the measure $\mu$ in Thm.~\ref{thm:mon} is bounded, we obtain
\begin{equation}
f(t) \leq \mu(\{0\}) + 2 \int_{(0,1/2]} \total \mu(\lambda) + t \left[ \mu(\{1\}) + 2 \int_{(1/2,1)} \total \mu(\lambda) \right] \eqend{,}
\end{equation}
such that $f(t)$ is bounded by a linear function, say $f(t) \leq a + b t$ with $a,b \geq 0$. Moreover, the dominated convergence theorem shows that
\begin{equation}
f'_\infty = \lim_{t \to \infty} \frac{f(t)}{t} = \int_{(0,1)} \lim_{t \to \infty} \frac{1}{\lambda + (1-\lambda) t} \total \mu(\lambda) + \mu(\{1\}) = \mu(\{1\}) \eqend{.}
\end{equation}
Hence, it follows that
\begin{align}
&\int_0^\infty f(t) \total \norm{ E^{\Delta_{\varphi,\psi}}(t) \xi_\psi }^2 + f'_\infty \varphi\left( \1 - s(\psi) \right) \\
&\quad\leq a \int_0^\infty \total \norm{ E^{\Delta_{\varphi,\psi}}(t) \xi_\psi }^2 + b \int_0^\infty t \total \norm{ E^{\Delta_{\varphi,\psi}}(t) \xi_\psi }^2 + f'_\infty \varphi\left( \1 - s(\psi) \right) \\
&\quad= a \psi(\1) + b \varphi\left( s(\psi) \right) + f'_\infty \varphi\left( \1 - s(\psi) \right) \eqend{,}
\end{align}
and the integral is finite. Using Thm.~\ref{thm:mon}, Tonelli's theorem, and Lemma~\ref{lemma:inf} with $t = \lambda/(1-\lambda)$, we therefore obtain
\begin{align}
&\int_0^\infty f(t) \total \norm{ E^{\Delta_{\varphi,\psi}}(t) \xi_\psi }^2 + f'_\infty \varphi\left( \1 - s(\psi) \right) - \mu(\{0\}) \psi(\1) - \mu(\{1\}) \varphi(\1) \\
&\quad= \int_0^\infty \int_{(0,1)} \frac{t}{\lambda + (1-\lambda) t} \total \mu(\lambda) \total \norm{ E^{\Delta_{\varphi,\psi}}(t) \xi_\psi }^2 \\
&\quad= \int_{(0,1)} \left( \xi_\psi, \frac{\Delta_{\varphi,\psi}}{\lambda + (1-\lambda) \Delta_{\varphi,\psi}} \xi_\psi \right) \total \mu(\lambda) \\
&\quad= \int_{(0,1)} \inf_{x \in \M} \left[ \frac{\psi\left( (\1-x^*) (\1-x) \right)}{1-\lambda} + \frac{\varphi\left( x x^* \right)}{\lambda} \right] \total \mu(\lambda) \eqend{.}
\end{align}
Using then Prop.~\ref{prop:ineq}, it follows that
\begin{align}
&\int_0^\infty f(t) \total \norm{ E^{\Delta_{\varphi,\psi}}(t) \xi_\psi }^2 + f'_\infty \varphi\left( \1 - s(\psi) \right) - \mu(\{0\}) \psi(\1) - \mu(\{1\}) \varphi(\1) \\
&\quad\geq \frac{1}{2} \Bigl( \varphi(\1) + \psi(\1) - \norm{ \varphi - \psi } \Bigr) \int_{(0,1)} \total \mu(\lambda) \eqend{,}
\end{align}
and since also
\begin{align}
\frac{1}{2} \Bigl( \varphi(\1) + \psi(\1) - \norm{ \varphi - \psi } \Bigr) = \psi(\1) - ( \varphi - \psi )_-(\1) \leq \psi(\1) \eqend{,} \\
\frac{1}{2} \Bigl( \varphi(\1) + \psi(\1) - \norm{ \varphi - \psi } \Bigr) = \varphi(\1) - ( \varphi - \psi )_+(\1) \leq \varphi(\1) \eqend{,}
\end{align}
we obtain the conclusion:
\begin{align}
&\int_0^\infty f(t) \total \norm{ E^{\Delta_{\varphi,\psi}}(t) \xi_\psi }^2 + f'_\infty \varphi\left( \1 - s(\psi) \right) \\
&\quad\geq \frac{1}{2} \Bigl( \varphi(\1) + \psi(\1) - \norm{ \varphi - \psi } \Bigr) \left[ \mu(\{0\}) + \mu(\{1\}) + \int_{(0,1)} \total \mu(\lambda) \right] \\
&= \frac{f(1)}{2} \Bigl( \varphi(\1) + \psi(\1) - \norm{ \varphi - \psi } \Bigr) \eqend{.}
\end{align}
\end{proof}
It remains to show the Proposition.
\begin{proof}[Proof of Prop.~\ref{prop:ineq}]
Consider the projection $p \coloneq s(\varphi + \psi) = s(\varphi) \vee s(\psi)$, where $s(\varphi)$ and $s(\psi)$ are the support projections of $\varphi$ and $\psi$ in $\M$. Since
\begin{equation}
\varphi\left( x x^* \right) = \varphi\left( p x x^* p \right) = \varphi\left( p x p x^* p \right) + \varphi\left( p x (\1-p) x^* p \right) \geq \varphi\left( p x p (p x p)^* \right)
\end{equation}
and analogously $\psi\left( (\1-x^*) (\1-x) \right) \geq \psi\left( [ p (\1-x) p ]^* p (\1-x) p \right)$, it is enough to prove the inequality for $p \M p$ instead of $\M$. Let $\chi \coloneq \varphi + ( \varphi - \psi )_- = \psi + ( \varphi - \psi )_+$. Because $2 \chi \geq \varphi + \psi$, we may thus assume that $\chi$ is faithful, such that the representative vector $\xi_\chi \in \mathcal{P}$ of $\chi$ in the positive cone is cyclic and separating for $\M$. We use the convention that the scalar product in $\mathcal{H}$ is linear in the second entry.

For any $\rho \in \M\stp$ with $\rho \leq \chi$, let $\xi_\rho \in \mathcal{P}$ be its representative vector in the positive cone. The map $V_\rho$ defined by $V_\rho( x \xi_\chi ) \coloneq x \xi_\rho$ for $x \in \M$ is contractive:
\begin{equation}
\norm{ V_\rho( x \xi_\chi ) }^2 = \norm{ x \xi_\rho }^2 = \rho\left( x^* x \right) \leq \chi\left( x^* x \right) = \norm{ x \xi_\chi }^2 \eqend{,}
\end{equation}
and since $\M \xi_\chi$ is dense in $\mathcal{H}$, it extends by continuity to a contraction $V_\rho \colon \mathcal{H} \to \mathcal{H}$. On $\M \xi_\chi$, we compute
\begin{equation}
V_\rho( x y \xi_\chi ) = x y \xi_\rho = x ( V_\rho y \xi_\chi ) \eqend{,}
\end{equation}
such that $V_\rho$ commutes with the action of $\M$, and because it is a bounded operator on $\mathcal{H}$, it is thus an element of $\M'$. It follows that $T_\rho \coloneq V_\rho^* V_\rho \in \M'$ with $0 \leq T_\rho \leq \1$, and we obtain a linear map $\rho \mapsto T_\rho$ given by
\begin{equation}
\rho\left( x^* y \right) = \left( x \xi_\rho, y \xi_\rho \right) = \Bigl( V_\rho( x \xi_\chi ), V_\rho( y \xi_\chi ) \Bigr) = \left( x \xi_\chi, V_\rho^* V_\rho y \xi_\chi \right) = \left( x \xi_\chi, T_\rho y \xi_\chi \right) \eqend{.}
\end{equation}

In the following, we consider $T_\varphi$, $T_\psi$, $T_\chi$, and $T_\pm \coloneq T_{( \varphi - \psi )_\pm}$, which are related by
\begin{equation}
\1 = T_\chi = T_\varphi + T_- = T_\psi + T_+
\end{equation}
through the definition of $\chi$. Let $p_\pm \coloneq s( ( \varphi - \psi )_\pm )$ be the support projections of $( \varphi - \psi )_\pm$ in $\M$. Then we have
\begin{equation}
\norm{ T_\pm^\frac{1}{2} (\1 - p_\pm) \xi_\chi }^2 = \left( (\1 - p_\pm) \xi_\chi, T_\pm (\1 - p_\pm) \xi_\chi \right) = ( \varphi - \psi )_\pm\left( \1 - p_\pm \right) = 0 \eqend{,}
\end{equation}
and from this $0 = T_\pm (\1 - p_\pm) \xi_\chi = (\1 - p_\pm) T_\pm \xi_\chi$, or $T_\pm \xi_\chi = p_\pm T_\pm \xi_\chi$. Using the orthogonality $p_- p_+ = 0$ of the Jordan decomposition, it thus follows that
\begin{equation}
\left( T_- \xi_\chi, T_+ \xi_\chi \right) = \left( p_- T_- \xi_\chi, p_+ T_+ \xi_\chi \right) = \left( \xi_\chi, T_- p_- p_+ T_+ \xi_\chi \right) = 0 \eqend{.}
\end{equation}

Define now the vectors
\begin{equation}
\alpha \coloneq V_\psi T_\varphi \xi_\chi \eqend{,} \quad \beta \coloneq V_\varphi T_\psi \xi_\chi
\end{equation}
and for $x \in \M$ the functional
\begin{equation}
L(x) \coloneq \left( T_\psi x \xi_\chi, T_\varphi \xi_\chi \right) \eqend{.}
\end{equation}
Using that $V_\rho, T_\rho \in \M'$, a short computation yields
\begin{equation}
L(x) = \left( x \xi_\psi, \alpha \right) = \left( \beta, x^* \xi_\varphi \right) \eqend{.}
\end{equation}
Moreover, we obtain
\begin{align}
L(\1) &= \left( T_\psi \xi_\chi, T_\varphi \xi_\chi \right) = \left( ( \1 - T_+ ) \xi_\chi, ( \1 - T_- ) \xi_\chi \right) \\
&= \norm{ \xi_\chi }^2 - \left( \xi_\chi, T_+ \xi_\chi \right) - \left( \xi_\chi, T_- \xi_\chi \right) + \left( T_+ \xi_\chi, T_- \xi_\chi \right) \\
&= \chi(\1) - ( \varphi - \psi )_+(\1) - ( \varphi - \psi )_-(\1) \in \mathbb{R} \eqend{,}
\end{align}
\begin{align}
L(\1) - \norm{ \alpha }^2 &= \left( ( \1 - T_+ ) T_- \xi_\chi, ( \1 - T_- ) \xi_\chi \right) \\
&= \norm{ ( \1 - T_- )^\frac{1}{2} T_-^\frac{1}{2} \xi_\chi }^2 + \norm{ T_+^\frac{1}{2} T_- \xi_\chi }^2 - \left( T_- \xi_\chi, T_+ \xi_\chi \right) \geq 0 \eqend{,}
\end{align}
and
\begin{align}
L(\1) - \norm{ \beta }^2 &= \left( ( \1 - T_+ ) \xi_\chi, ( \1 - T_- ) T_+ \xi_\chi \right) \\
&= \norm{ ( \1 - T_+ )^\frac{1}{2} T_+^\frac{1}{2} \xi_\chi }^2 + \norm{ T_-^\frac{1}{2} T_+ \xi_\chi }^2 - \left( T_- \xi_\chi, T_+ \xi_\chi \right) \geq 0 \eqend{,}
\end{align}
where we used that $0 \leq T_\pm \leq \1$ and that $\left( T_- \xi_\chi, T_+ \xi_\chi \right) = 0$.

For $\lambda \in (0,1)$, it thus follows that
\begin{align}
0 &\leq \norm{ \lambda^\frac{1}{2} \beta - \lambda^{-\frac{1}{2}} x^* \xi_\varphi }^2 + \norm{ (1-\lambda)^{-\frac{1}{2}} (\1-x) \xi_\psi - (1-\lambda)^\frac{1}{2} \alpha }^2 \\
&= \lambda \norm{ \beta }^2 + \lambda^{-1} \norm{ x^* \xi_\varphi }^2 - \left( \beta, x^* \xi_\varphi \right) - \overline{ \left( \beta, x^* \xi_\varphi \right) } \\
&\quad+ (1-\lambda)^{-1} \norm{ (\1-x) \xi_\psi }^2 + (1-\lambda) \norm{ \alpha }^2 - \left( (\1-x) \xi_\psi, \alpha \right) - \overline{ \left( (\1-x) \xi_\psi, \alpha \right) } \\
&= \lambda \norm{ \beta }^2 + \lambda^{-1} \varphi\left( x x^* \right) + (1-\lambda)^{-1} \psi\left( (\1-x^*) (\1-x) \right) + (1-\lambda) \norm{ \alpha }^2 - 2 L(\1) \eqend{,}
\end{align}
and from this
\begin{align}
\frac{\varphi\left( x x^* \right)}{\lambda} + \frac{\psi\left( (\1-x^*) (\1-x) \right)}{1-\lambda} &\geq 2 L(\1) - \lambda \norm{ \beta }^2 - (1-\lambda) \norm{ \alpha }^2 \\
&\geq L(\1) = \chi(\1) - ( \varphi - \psi )_+(\1) - ( \varphi - \psi )_-(\1) \\
&= \frac{1}{2} \Bigl( \varphi(\1) + \psi(\1) - \norm{ \varphi - \psi } \Bigr)
\end{align}
as required.
\end{proof}

\end{document}